\documentclass[11pt,dvipsnames,reqno]{amsart}

\usepackage{times}
\usepackage[all]{xy}
\usepackage{amssymb,amsmath,amsthm}
\usepackage{graphicx}
\usepackage{amssymb}

\allowdisplaybreaks

\newcommand{\calO}{\mathcal{O}}
\newcommand{\calP}{\mathcal{P}}

\newcommand{\frakS}{\mathfrak{S}}
\newcommand{\frakG}{\mathfrak{G}}

\newcommand{\Sym}{\mathrm{Sym}}

\newcommand{\bbP}{\mathbb{P}}
\newcommand{\bbZ}{\mathbb{Z}}
\newcommand{\bbQ}{\mathbb{Q}}

\newcommand{\Pic}{\operatorname{Pic}}

\usepackage[
           colorlinks=true,
            urlcolor=blue,       
            filecolor=blue,     
            linkcolor=blue,       
            citecolor=blue,         
            pdftitle= {Title},
            pdfauthor={Author},
            pdfsubject={Subject},
            pdfkeywords={Key}
            pagebackref,
            pdfpagemode=None,
            bookmarksopen=true]
            {hyperref}

\makeatletter
\newtheorem{theorem}{Theorem}[section]
 \newtheorem{proposition}{Proposition}[section]
 \let\c@proposition\c@theorem
 
 \newtheorem{lemma}{Lemma}[section]
 \let\c@lemma\c@theorem
 
 \newtheorem{corollary}{Corollary}[section]
 \let\c@corollary\c@theorem

 \let\c@con\c@theorem

 \let\c@question\c@theorem
 
\theoremstyle{definition}
 \newtheorem{example}{Example}[section]
 \let\c@example\c@theorem
 
 \newtheorem{defn}{Definition}[section]
 \let\c@defn\c@theorem
 
 \newtheorem{remark}{Remark}[section]
 \let\c@remark\c@theorem
 
 \newtheorem{observation}{Observation}[section]
 \let\c@observation\c@theorem
 
 \newtheorem{convention}{Convention}[section]
 \let\c@convention\c@theorem
 
\let\c@equation\c@theorem

\makeatother
\def\sectionautorefname~{\S}

\newcommand{\beq}{\begin{equation}}
\newcommand{\eeq}{\end{equation}}
\newcommand{\la}{\langle}
\newcommand{\ra}{\rangle}

\def\Z{\mathbb{Z}}
\def\Q{\mathbb{Q}}
\def\C{\mathbb{C}}
\def\dual{^\vee}
\def\rk{\operatorname{rk}}

\def\Aut{\operatorname{Aut}}

\def\discr{\operatorname{discr}}

\def\rad{\operatorname{Ker}}

\def\Cp#1{\mathbb{P}^{#1}}
\def\pencil{\mathcal{P}}
\def\Hom{\operatorname{Hom}}

\def\id{\operatorname{id}}
\def\Sym{\operatorname{Sym}}

\def\rad{\operatorname{rad}}
\def\SG#1{\mathfrak{S}_{#1}}
\def\DG{\mathfrak{D}}

\def\2{\color{red}}
\def\3{\color{magenta}}

{
\catcode`\1\active
\catcode`\.\active
\gdef\tconfig{\vcenter\bgroup\offinterlineskip
\catcode`\1\active\catcode`\.\active\def1{&\bullet}\def.{&\cdot}%
\halign\bgroup
 \hbox to0pt{\hss##\hss}&&\hbox to9pt{\small\hss$\mathstrut##$\hss}\cr}
\gdef\endtconfig{\crcr\egroup\egroup}
}

\def\tA{\smash{\tilde{A}}}
\def\tD{\smash{\tilde{D}}}
\def\tE{\smash{\tilde{E}}}

\def\qa{\alpha}
\def\qb{\beta}

\def\bS{\bar S}
\def\bX{\bar X}
\def\bG{\bar G}
\def\bT{\bar T}
\def\bGamma{\bar\Gamma}

\def\lineno{\afterassignment\linenoii\count0=}
\def\linenoii{\let\next=L\ifnum\count0>12\let\next=M\advance\count0-12\fi
 \next_{\the\count0}}
\def\DEFS{\let\*\times\def\''{^*}\let\l\lineno\let\c\mathbf}%

\def\ii{\sqrt{-1}}
\def\quadric{\mathcal{Q}}
\def\plane{R}

\def\cone{\mathcal{I}}
\def\cubic{\mathcal{C}}

\title{K3 surfaces of degree six arising from desmic tetrahedra}
\author{Alex Degtyarev}

\address{%
Bilkent University\\
Department of Mathematics\\
06800 Ankara, Turkey}

\email{degt@fen.bilkent.edu.tr}

\thanks{%
The first author was partially supported by the
T\''{U}B\.{I}TAK grant
123F111.}

\author{Igor Dolgachev}
\address{Department of Mathematics, University of Michigan, 525 East University Avenue, Ann Arbor, MI 48109-1109 USA}
\email{idolga@umich.edu}

\author{Shigeyuki Kond\=o}
\address{Graduate School of Mathematics, Nagoya University, Nagoya, 464-8602 Japan}
\email{kondo@math.nagoya-u.ac.jp}
\thanks{Research of the third author is partially supported by JSPS Grant-in-Aid for Scientific Research (A) No.20H00112, (B) No.25K00906.}



\begin{document}

\begin{abstract}
We study K3 surfaces of degree 6 containing two sets of 12 skew lines such that each line from
a set intersects exactly six lines from
the other set.
These surfaces arise as hyperplane sections of the cubic line complex associated with the
pencil of desmic quartic surfaces introduced by George Humbert and recently studied by the second
and
third authors. We
discuss alternative birational
models of the surfaces, compute the Picard lattice
and a group of projective automorphisms,
and describe rational curves of low degree on the general surface.
\end{abstract}

\maketitle

\section{Introduction}\label{into}
Three tetrahedra in projective space $\bbP^3$ are called \emph{desmic} if any
two of them are perspective with respect to any vertex
of the third one.
Equivalently, three tetrahedra are desmic if they can be included in a pencil of quartic surfaces, called a desmic pencil.
An irreducible member of a desmic pencil is called a \emph{desmic quartic surface}.
It contains 16 lines,
the base locus of the pencil, and 12 nodes lying by pairs in the edges of any desmic tetrahedron.
Each line passes through 3 nodes, and each node is contained
in four lines. A desmic pencil defines an associated desmic pencil such that the
original twelve nodes
are the vertices of
the new desmic
tetrahedra. George Humbert, in his study of desmic quartic surfaces \cite{Humbert}, showed that the set of lines contained
in a quadric passing through the vertices of any two of the desmic tetrahedra in
the associated desmic pencil is a cubic line complex $\mathfrak{G}$ in the
Grassmannian $G_1(\bbP^3)$ of lines in $\bbP^3$;
it does not depend on the choice of
the pair of tetrahedra used.
The complex $\mathfrak{G}$
contains 24 planes, twelve from each family of planes in the Pl\"{u}cker embedding of $G_1(\bbP^3)$
in $\bbP^5$.

The subject of the present paper is a transversal hyperplane section $X$ of $\mathfrak{G}$, which is a smooth K3 surface equal
to the complete intersection of a quadric and a cubic hypersurface in $\bbP^4$.
 In this paper, we call $X$ a Humbert sextic K3
surface (this should not be confused with the Humbert surfaces
from the theory of abelian surfaces with special properties of their endomorphism ring, see \cite[Chap. IX]{vGeer}).
It contains two sets of twelve lines whose incidence relation
is an abstract configuration $(12_6,12_6)$.
It will be shown elsewhere
 by the first author
that other sextic K3 surfaces
cannot contain such a configuration of lines.

Projecting from any line, we find a birational model of $X$ isomorphic to a double cover of $\bbP^2$ branched along a plane sextic $B$ with nodes at the
vertices of a complete quadrilateral $P$, see \autoref{planemodel}.
The curve $B$ is
contact (i.e., has an even intersection index at each intersection point,
which are all distinct from the nodes) to
the diagonals of $P$, a nodal plane cubic, and six lines in general linear position. Although the ten
 tangency conditions seem to impose too many
constraints on the existence of the curve $B$, of which we have a $5$-dimensional family,
in \autoref{s.plain}.
 We prove that, in fact, the conditions of the tangency
to the diagonals and one line
almost imply (modulo some finite combinatorial choices)
all other conditions. This was made possible
by the computation of the Picard lattice of a general
Humbert sextic $X$: it turns out to be freely generated by 15 lines,
see \autoref{s.NS}, \autoref{trans}.

 Among other things dealt with
in the paper is a description of smooth rational curves of degree
at most $4$
(see \autoref{s.lines} and \autoref{s.rational})
and elliptic pencils
(see \autoref{s.elliptic})
on a very general Humbert sextic K3 surface $X$.
We also discuss the
groups of projective and birational automorphisms of $X$,
see \autoref{s.aut}. Most notably, we show that any projective
automorphism of $X$ is induced from one of the Humbert line complex~$\frakG$;
in view of~\cite{DolgKondo}, this gives us a complete description of such
automorphisms (\autoref{prop.projective}).

 We work over the field of complex numbers, however, many of our results are valid assuming only that the ground field is an algebraically closed field
of characteristic
other than $2$ or $3$.

\section{Desmic pencils and the Humbert cubic line complex}\label{s.intro}
In this section, we introduce  Humbert sextic K3 surfaces
and discuss the configuration $(12_6,12_6)$ of lines on them.

\subsection{The line complex}\label{s.complex}
Consider three tetrahedra in $\bbP^3$:
\begin{equation}
\alignedat2
&T_1:\ &(x^2-y^2)(z^2-w^2) &= 0,\\
&T_2:\ &(x^2-z^2)(y^2-w^2) &= 0,\\
&T_3:\ &(w^2-x^2)(z^2-w^2) &= 0.
\endaligned
\end{equation}
They are desmic, i.e., belong to the same pencil
$$aT_1+bT_2+cT_3 = 0,\quad a+b+c = 0.$$
Any other member of the pencil is a \emph{desmic quartic surface} with 12 nodes
given in \eqref{points}, lying by pairs on the
edges of each of the tetrahedra (see \cite{DolgKondo}).
It also contains 16 lines lying by four on each face of
each
tetrahedron and intersecting the edges
at the singular points. This forms a configuration $(12_4,16_3)$ of nodes and lines.

Consider the following twelve points $P_i$ in $\bbP^3$:
\begin{equation}\label{points}
\def\no#1{{\scriptstyle#1}\colon}
\def\sp{\ \,}
\alignedat4
\no1&[0,0,0,1],\sp&\no2&[0,0,1,0],\sp& \no3&[0,1,0,0],\sp& \no4& [1,0,0,0],\\
\no5&[1,1,1,1],\sp&\no6&[1,-1,1,-1],\sp& \no7&[1,1,-1,-1],\sp& \no8& [1,-1,-1,1],\\
\no9&[1,-1,1,1],\sp&\no{10}&[1,1,-1,1],\sp& \no{11}&[1,1,1,-1],\sp& \no{12}& [-1,1,1,1].
\endalignedat
\end{equation}
They lie in pairs on the edges of any of the desmic tetrahedra with
the faces $\Pi_j$:
\begin{equation}\label{planes}
\def\no#1{{\scriptstyle#1}\colon}
\alignedat8
\no1&&x+y&= 0,\quad& \no2&&x-y&=0,\quad&\no3&&z+w&=0,\quad&\no4&&z-w&=0,\\
\no5&&y-w&= 0,\quad& \no6&&x-z&=0,\quad&\no7&&y+w&=0,\quad&\no8&&x+z&=0,\\
\no9&&y+z&= 0,\quad&\no{10}&&x-w&=0,\quad&\no{11}&&x+w&=0,\quad&\no{12}&&y-z&=0.
\endalignedat
\end{equation}
Each of the twelve points lies in six planes. Each plane contains six points.

Each pair of points
in the
same row of~\eqref{points} lies in two planes
from~\eqref{planes}, and each pair of planes
from the same row of~\eqref{planes} contains two common points
from~\eqref{points}.
The precise incidence relations are illustrated in \autoref{table}.
\begin{table}[htbp]
\begin{center}
\caption{The incidence relation between $12+12$ lines}\label{table}
\hrule height0pt
\def\rb#1{\rotatebox[origin=c]{90}{\ $#1$\ }}
\def\cb#1{\hbox to11pt{\hss#1\hss}}
\def\vr{\vrule width0pt height11pt}
\def\*{\rlap{$^*$}}
\scalebox{1}{%
\begin{tabular}{|c|c|cccc|cccc|cccc|}\hline
&&\rb{x+y}&\rb{x-y}&\rb{z+w}&\rb{z-w}&\rb{y-w}&\rb{x-z}&\rb{y+w}&\rb{x+z}&
 \rb{y+z}&\rb{x-w}&\rb{x+w}&\rb{y-z}\\\hline
&&\vr\cb{1\*}&\cb{2\*}&\cb{3}&\cb{4}&\cb{5}&\cb{6\*}&\cb{7}&\cb{8\*}&\cb{9\*}&\cb{10}&\cb{11}&\cb{12\*}\\
 \hline\vr
$[0,0,0,1]$&1\*&$\bullet$&$\bullet$&&&&$\bullet$&&$\bullet$&$\bullet$&& &$\bullet$ \\
$[0,0,1,0]$&2&$\bullet$&$\bullet$&&&$\bullet$&&$\bullet$&&&$\bullet$&$\bullet$&\\
$[0,1,0,0]$&3&&&$\bullet$&$\bullet$&&$\bullet$& &$\bullet$&&$\bullet$ &$\bullet$&\\
$[1,0,0,0]$&4&&&$\bullet$&$\bullet$&$\bullet$&&$\bullet$&&$\bullet$&&&$\bullet$\\
 \hline\vr
$[1,1,1,1]$&5&&$\bullet$&&$\bullet$&$\bullet$&$\bullet$&&&  &$\bullet$&&$\bullet$\\
$[1,-1,1,-1]$&6&$\bullet$&&$\bullet$&&$\bullet$&$\bullet$&&&$\bullet$&&$\bullet$&\\
$[1,1,-1,-1]$&7&&$\bullet$&&$\bullet$&&&$\bullet$&$\bullet$&$\bullet$&&$\bullet$&\\
$[1,-1,-1,1]$&8&$\bullet$&&$\bullet$&&&&$\bullet$&$\bullet$&&$\bullet$&&$\bullet$\\
 \hline\vr
$[1,-1,1,1]$&9&$\bullet$&&&$\bullet$&&$\bullet$ &$\bullet$&&$\bullet$&$\bullet$&&\\
$[1,1,-1,1]$&10&&$\bullet$&$\bullet$&&$\bullet$&&&$\bullet$&$\bullet$&$\bullet$& &\\
$[1,1,1,-1]$&11&&$\bullet$&$\bullet$&&&$\bullet$ &$\bullet$&&&&$\bullet$&$\bullet$\\
$[-1,1,1,1]$&12&$\bullet$&&&$\bullet$&$\bullet$&&&$\bullet$&&&$\bullet$&$\bullet$ \\ \hline
\end{tabular}}
\end{center}
\end{table}

George Humbert \cite{Humbert}
constructed a line complex of degree 3, i.e., a hypersurface $\mathfrak{G}$ in the
Grassmannian $G_1(\bbP^3) \subset \bbP^5$ cut out by a cubic hypersurface in $\bbP^5$
whose rays are lines in quadrics passing
through
the eight points forming any two rows in \eqref{points}. It contains
\begin{equation}
\aligned
&\text{$12$\  $\alpha$-planes of rays through each of the
points
$P_1,\ldots, P_{12}$ from \eqref{points}},\\
&\text{$12$\ $\beta$-planes
of rays lying in each plane $\Pi_1,\ldots,\Pi_{12}$ from \eqref{planes}}.
\endaligned
\label{eq.alpha-beta}
\end{equation}
The cubic complex  $\mathfrak{G}$ has 34 isolated singular points, 16 of which are lines on a desmic quartic surface, and
18 are the edges of all three tetrahedra.

\begin{defn} A \emph{Humbert sextic K3 surface} $X$ is a transversal hyperplane section of
$\mathfrak{G}$.
\end{defn}

\begin{remark}\label{rem.tangent}
Thus, we assume that $X$ is cut off $\frakG$ by a
hyperplane~$\plane$
transversal to~$\frakG$, but not necessarily to $\quadric=G_1(\bbP^3)$.
Note though that, even if $\plane$ is tangent to~$\quadric$,
none of the $24$ planes~\eqref{eq.alpha-beta} lies in~$\plane$, as otherwise
$\frakG$ would contain the tangency point and $\plane$ would not be
transversal to $\frakG$.
\end{remark}

\begin{convention}\label{conv.alpha-beta}
In view of \autoref{rem.tangent}, the hyperplane~$\plane$
cuts each $\alpha$-plane
corresponding to a point $P_i$
and each $\beta$-plane corresponding to a plane $\Pi_i$
see~\eqref{eq.alpha-beta}, along a line.
If $P_i\in\Pi_j$,
the corresponding
$\alpha$-plane and $\beta$-plane intersect along the line consisting of rays passing
through $P_i$ and lying on $\Pi_j$.
This gives us twelve \emph{$\alpha$-lines} $L_1,\ldots,L_{12}$ and twelve
\emph{$\beta$-lines} $M_1,\ldots,M_{12}$
on $X$ forming a configuration $(12_6,12_6)$;
we use the same numbering
for the lines as that for the
planes, so that their
incidence matrix is also given by \autoref{table}.

In the table, each of the $12$-tuples $\alpha$, $\beta$ is divided into three groups of four,
called \emph{quartets}:
$$
\alpha=\qa_1\cup\qa_2\cup\qa_3,\quad
\beta=\qb_1\cup\qb_2\cup\qb_3.
$$

From now on, we reserve $H$ for the divisor class of a hyperplane section
of~$X$, and we identify the $\alpha$-lines $L_1,\ldots,L_{12}\in \alpha$ and
$\beta$-lines $M_1,\ldots,M_{12}\in \beta$
(as well as other irreducible curves $C$ with $C^2<0$)
with their divisor classes.
\end{convention}

\subsection{Equation}\label{s.eqns}
One can replace the Pl\"{u}cker coordinates in the Pl\"{u}cker embedding
$G_1(\bbP^3)\hookrightarrow \bbP^5$ with the Klein coordinates to get the following
explicit equation of the Humbert cubic line complex (see \cite[the equation (12)]{DolgKondo}):
\beq\label{eqCC}
\begin{split}
&Q:= x_1^2+x_2^2+x_3^2+y_1^2+y_2^2+y_3^2 = 0,\\
&F:= x_1x_2x_3+\ii y_1y_2y_3 = 0.
\end{split}
\eeq
So, the equation of the Humbert
sextic K3
surface in $\bbP^5$ is obtained by adding
an extra
equation
\begin{equation}
\plane= a_1x_1+a_2x_2+a_3x_3+a_4y_1+a_5y_2+a_6y_6 = 0.
\label{eq.HS}
\end{equation}
The condition
that the surface is singular defines a closed
subset of $\bbP^4$ (for example, containing the intersection of two coordinate hyperplanes $a_i=a_j = 0$).

The set $\calP$ of
the
24 planes is naturally indexed by
the
elements of the
symmetric group
$\frakS_4\cong (\bbZ/2\bbZ)^2 \rtimes \frakS_3$:
interpreting $(\bbZ/2\bbZ)^2$
as $\{\pm\ii\}^3/\epsilon_1\epsilon_2\epsilon_3$,
the plane corresponding to
$\sigma\in\SG3$ and $(\epsilon_1,\epsilon_2,\epsilon_3)\in(\bbZ/2\bbZ)^2$
is
\beq\label{Planes2}
V(x_1-\epsilon_1y_{\sigma(1)},x_2-\epsilon_2y_{\sigma(2)},
x_3-\epsilon_3y_{\sigma(3)}).
\eeq
The $\alpha$- and $\beta$-planes differ by the sign of $\sigma$.

Note that equation \eqref{eqCC} is invariant with respect to the natural action of the
group $\frakS_4\times \frakS_4\cong  (\bbZ/2\bbZ)^4 \rtimes (\frakS_3\times \frakS_3)$
(and it is this action that defines the group structure on the set
theoretic Cartesian product of $(\Z/2\Z)^2$ and $\SG3$ above). It is also invariant with respect to
the transformation $g_0$ of order $4$ defined by
$$(x_1,x_2,x_3,y_1,y_2,y_3)\mapsto (-y_1,-y_2,-y_3,x_1,x_2,x_3).$$
It is proved in \cite{DolgKondo} that the group $G = (\frakS_4\times \frakS_4)\rtimes \bbZ/2\bbZ$
of order $1152$ is the
full
group of projective automorphisms
of the cubic line complex $\frakG$.

\subsection{Lines, conics, quartics}\label{s.lines}
The K3 surface $X$ is a transversal intersection of a quadric
$\quadric\cap(\text{hyperplane})$ and a cubic, where $\quadric =V(Q)$.
A hyperplane section of $X$ is a curve of bidegree $(3,3)$ on a quadric.
It follows from \autoref{table} that  there are $3\binom{4}{2} = 18$  hyperplane sections
that contains a
\emph{quadrangle} of lines, two from each ruling of the quadric,
and an irreducible residual conic.
Each of these quadrangles, called \emph{proper},
is determined by its
pair of $\alpha$-lines (or pair of $\beta$-lines) and can be characterised by
the property that its $\alpha$-lines (resp.\ $\beta$-lines) are in the same
quartet.

Each
line belongs to 3 proper quadrangles, indexed by the quartets in the opposite
family.
Thus, the $24$ lines and $18$
proper quadrangles form an abstract configuration
$(24_3,18_4)$. Each proper quadrangle determines a pair $\qa_r,\qb_s$ of
quartets, i.e., a $(4\times4)$-cell in \autoref{table}. Each cell is the
union of
two proper quadrangles.

There are  $16$ hyperplane sections that contain
a union of 6 lines, three from each family $\alpha,\beta$.
Its dual graph is
a complete bipartite graph $K(3,3)$.
For this reason, we call such unions of lines \emph{$(3,3)$-configurations} of
lines, or just \emph{$(3,3)$-fragments} (of the full configuration of lines
of a given K3 sextic).
The $\alpha$- (resp.\ $\beta$-) lines of each $(3,3)$-configuration are in
a bijection with the $\alpha$- (resp.\ $\beta$-) quartets.
The $16$ $(3,3)$-configurations and
$24$ lines
form an abstract configuration $(16_6,24_4)$. It follows that
$16H \sim 4\sum_{i=1}^{12}(L_i+M_i)$, hence,
\beq\label{mod4}
\sum_{i=1}^{12}(L_i+M_i) \sim 4H.
\eeq

Each $(3,3)$-fragment consists of nine
\emph{improper}
quadrangles (with all $144=16\times9$ quadrangles pairwise
distinct), which are characterised by the property that their $\alpha$-lines
(resp.\ $\beta$-lines) are in distinct quartets.
In more detail, $(3,3)$-configurations on smooth sextic K3 surfaces are
discussed in \autoref{rem.3,3.property} below.

For each quadrangle~$q$, we denote by $\sum q$ the sum of
the four lines in $q$.
Thus,
each quadrangle~$q$ gives rise to a residual conic $C_q\in|H-\sum q|$;
the latter
is irreducible if and only if $q$ is proper.
If the $\alpha$-lines (hence, also $\beta$-lines) of two proper quadrangles
$q_1$, $q_2$ constitute a whole quartet, $q_1$ and $q_2$ are disjoint and,
hence, constitute two singular fibers of a common elliptic pencil. It follows
that $\sum q_1=\sum q_2$ and $C_1=C_2$.
This common conic is the intersection $\mathcal{Q}\cap H_1\cap H_2$,
where $H_i\subset\Cp4$ is the hyperplane spanned by $q_i$.

\begin{observation}\label{rem.conics}
We
conclude that the conics on~$X$ (at least those accounted for so far) can
be indexed by
\begin{enumerate}
\item\label{index.a}
pairs $L_i,L_j$ of $\alpha$-lines that are in the same quartet, or
\item\label{index.b}
pairs $M_i,M_j$ of $\beta$-lines that are in the same quartet, or
\item\label{index.qu}
pairs $(\qa_r,\qb_s)$ of quartets, i.e., the nine $(4\times4)$-cells in \autoref{table}.
\end{enumerate}
The first two indexing schemes are two-to-one, so we end up with the
$9$ conics $C_{rs}$, $1\le r,s\le3$, rather
than the expected $18$.

Indexed as in~\eqref{index.qu}, one has
$C_{rs}\cdot L_i=1$ if and only if $L_i\in\qa_r$ and
$C_{rs}\cdot M_j=1$ if and only if $M_j\in\qb_s$;
otherwise, $C_{rs}\cdot N=0$ for a line~$N$.
Besides, $C_{rs}\cdot C_{uv}=2$ if $r\ne u$, $s\ne v$ or $0$ otherwise.
For this reason, \eqref{index.qu} is the preferred indexing scheme.
\end{observation}

\begin{remark}\label{equationconic}
One can see the 9 conics and 18 proper quadrangles from~\eqref{eqCC}.

The cubic hypersurface $V(F)$ contains nine
three-dimensional subspaces $V(x_i,y_j)$,
each intersecting the quadric $\mathcal{Q}$ along a
quadric surface. Intersecting the latter by
$V(R)$, we obtain nine conics $C_{ij}$ contained in $X$.  Each
$C_{ij}$ is contained in two hyperplanes $H_{ij}^{\pm}  =  V(x_i\pm
iy_j)$. The hyperplane $H_{ij}^{\pm}$ cuts
$X$ along a quadrangle cut out by the planes
\eqref{Planes2}
with $\sigma(i) = j$. This gives us
the
18 proper quadrangles of lines.
\end{remark}

\begin{observation}\label{rem.quartics}
If $L_i\in\alpha$ and $M_j\in \beta$ are skew lines from distinct families,
the residual curve $Q_{ij}\in|H-L_i-M_j|$ is a smooth rational
quartic curve. This gives us
$6\times 12  = 72$ quartics, all distinct.
The lines $L_i$, $M_j$ are singled out via
$Q_{ij}\cdot L_i=Q_{ij}\cdot M_j=3$ whereas $Q_{ij}\cdot N\in\{0,1\}$
for any other line~$N$.
\end{observation}

In \autoref{s.rational} below we assert that, apart from the $24$ lines, $9$
conics, and $72$ quartics described in this section, a general Humbert sextic
has no smooth rational curves of degree up to~$4$.
There are infinitely many other smooth rational curves, see
\autoref{s.aut}.

\begin{observation}\label{obs.group}
We use
the \texttt{digraph} package in \texttt{GAP}~\cite{GAP4.13} to compute the group
$G=\Sym(\Gamma)$ of symmetries of the
dual adjacency graph~$\Gamma$ of lines on~$X$:
one has
$|G|=1152=16\times72$, and the group is generated by the involution
$L_i\leftrightarrow M_i$, $i=1,\ldots,12$, and two permutations
$$ L_i\mapsto L_{\sigma_1(i)},\quad M_i\mapsto M_{\sigma(i)},\quad  \sigma = (1,9,4,12)(2,10,3,11)(5,8,6,7),$$
and
$$\alignedat2
L_i&\mapsto L_{\sigma(i)}, &\quad \sigma &= (1,10,8,2,11,6)(3,9,7)(4,12,5),\\
M_i&\mapsto M_{\sigma(i)}, &\quad \sigma &= (1,2,3)(5,12)(6,9,8,10,7,11).
\endalignedat
$$
The group acts transitively on the set of the
$(3,3)$-fragments, and the stabilizer of a $(3,3)$-fragment $q$ is
isomorphic to the full group
$\Sym(q)=(\mathfrak{S}_3\times\mathfrak{S}_3)\rtimes\Z/2$.
In particular, $G$ is transitive. Alternatively, $G$ induces the full
automorphism group $(\mathfrak{S}_3\times\mathfrak{S}_3)\rtimes\Z/2$ on the
set of the $(4\times4)$-cells in \autoref{table}.
As an abstract graph, each cell~$c$ is the disjoint union of two quadrangles
and the stabilizer of~$c$ in $G$ maps two-to-one onto the index~$2$ subgroup
of $\Sym(c)=(\DG_8\times\DG_8)\rtimes\Z/2$
that is not mixing $\alpha$- and $\beta$-lines.

A posteriori one can easily verify that $\Sym(\Gamma)$ is indeed as
claimed:
the three permutations indicated do belong to $\Sym(\Gamma)$
and,
given \autoref{table}, it is immediate that any
$g\in\Sym(\Gamma)$ fixing pointwise a certain $(3,3)$-fragment~$q$ is the
identity.
\end{observation}

\begin{remark}\label{rem.3,3.property}
\let\L=A\let\M=B%
The configuration of lines on any smooth
sextic K3 surface $X\subset\Cp4$ (not necessarily the one considered in this paper)
has the following \emph{$(3,3)$-property}: given five distinct lines
$\L_1,\L_2,\L_3$ and $\M_1,\M_2$, such that $\L_i\cdot \M_k=1$ for all $i$,
$k$, there is a unique sixth line $\M_3$ such that $q=\{\L_1,\ldots,\M_3\}$
constitute
a $(3,3)$-configuration. Furthermore, any line on~$X$ that is not in~$q$
intersects exactly one line in~$q$. Arithmetically, $\M_3$ is found from
$$
\L_1+\cdots+\M_3=H\quad\text{in}\quad \Pic(X).
$$
Geometrically, once the residual conic~$C$ above splits, its two components
are in the two distinct rulings of the quadric. Note that we do
not even need to assume beforehand that the $\L$-lines or $\M$-lines are
pairwise disjoint. Should there be $n>0$ extra intersection points, the class
$$
e=H-(\L_1+\L_2+\L_3+\M_1+\M_2)\in\Pic(X)
$$
would have
$e^2=2n-2$ and $e\cdot H=1$. If $n=1$,
then, since $H$ is ample,
$e$ must be the class of an irreducible curve of arithmetic genus one, and
$|H|$ restricted to the curve has a base point, contradicting
\cite[Theorem 3.1]{Saint-Donat}.
If $n\ge2$,
the sublattice $\Z H+\Z e$
is positive definite, contradicting the Hodge index theorem.

In particular, it follows also that two lines $\M_1,\M_2$
cannot meet more
than three common lines $\L_i$. A similar argument shows that two
\emph{intersecting} lines $\M_1,\M_2$ can meet at most one common line.
\end{remark}

\section{Double plane model}\label{planemodel}

Let us consider the
projection
$$f\colon X\to \bbP^2$$
of $X$ with center at some $\alpha$-line or $\beta$-line, say $L = L_1$. It is given by the linear system
$|H-L|$. Its restriction
to $L$ is a hyperplane in $|\calO_L(3)|\cong \bbP^3$,
which has no base points;
hence, $f$ is a regular map.

We split $\beta$ into the complementary subsets
$$
\beta^*_1=\{M_1,M_2,M_6, M_8, M_9,M_{12}\},\quad
\bar\beta^*_1=\{M_3,M_4,M_5,M_7,M_{10},M_{11}\}
$$
of the lines
that, respectively, intersect or are disjoint from~$L$; the former are
marked with a $*$ in \autoref{table}.
(The subscript $1$ refers to the chosen line $L=L_1$.)

The standard formula for the canonical class of a double cover
shows that the branch curve $B$ of $f$ is of degree $6$.
The lines  $M_i\in\beta^*_1$ intersecting $L$ are blown down to the
nodes $p_1,\ldots,p_6$ of $B$.

It follows from \autoref{table} that
\begin{itemize}
\item $M_2,M_6,M_{12}$ intersect $L_5,L_{11}$,
\item $M_1, M_6, M_9$ intersect $L_6,L_{9}$,
\item  $M_2,M_8,M_9$ intersect $L_7,L_{10}$,
\item  $M_1,M_8,M_{12}$ intersect  $L_8,L_{12}$,
\item $M_1,M_2$ intersect $L_2$,
\item $M_6,M_8$ intersect $L_3$,
\item $M_9,M_{12}$ intersect $L_4$.
\end{itemize}
Since $(H-L)\cdot L_j = 1$ for any $j\ne 1$,
the images $f(L_j)$ are lines in the plane.
Each of the pairs
$(L_5,L_{11})$, $(L_6,L_{9})$, $(L_7,L_{10})$, $(L_8,L_{12})$ is mapped to the same line passing
through three of the nodes. Thus, their images form a complete quadrilateral with vertices
$p_1,\ldots,p_6$ and sides
$$
\alignedat2
\ell_{236} &= \la p_2,p_3,p_6\ra,\quad
&\ell_{135} &= \la p_1,p_3,p_5\ra, \\
\ell_{245} &= \la p_2,p_4,p_5\ra, \quad
&\ell_{146} &= \la p_1,p_4,p_6\ra.
\endalignedat
$$
The remaining $\alpha$-lines $L_2,L_3,L_4$  are mapped to the diagonals
$$
\ell_{12} = \la p_1,p_2\ra, \quad \ell_{34} = \la p_3,p_4\ra, \quad \ell_{56} = \la p_5,p_6\ra.
$$

The lines $L_i$, $i=2,3,4$, are in the same quartet with~$L_1$, thus giving
rise to conics $C_{1i}$ (see \autoref{rem.conics}). These conics
are
mapped to the diagonals, i.e.,
$$
\aligned
f^{-1}(\ell_{12}) &= M_1+M_2+L_2+C_{12},\\
f^{-1}(\ell_{34}) &= M_6+M_8+L_3+C_{13},\\
f^{-1}(\ell_{56}) &= M_9+M_{12}+L_4+C_{14}.
\endaligned
$$

For the remaining lines
$M_i\in\bar\beta^*_1$, since $(H-L)\cdot M_i = 1$,
the images
$$
\alignedat3
\ell_1 &= f(M_3), \quad  &\ell_2 &= f(M_4),\quad &\ell_3 &= f(M_5), \\
\ell_4 &= f(M_7), \  &\ell_5 &= f(M_{10}),\ &\ell_6 &= f(M_{11})
\endalignedat
$$
are lines. The quartic $Q_{1i}$ (see \autoref{rem.quartics}) is mapped to the same line
as~$M_i$.
Since
all $\beta$-lines are skew, none of these lines $\ell_k$
passes through any of the
nodes
$p_1,\ldots,p_6$. Each line
splitting under the cover,
 cuts out an even divisor $2d$ on the branch
curve~$B$.
In other words, it is a \emph{tritangent} (or \emph{contact line}) to~$B$.

\medskip
Now, let us look at the image of $L$ under the projection.
Since $(H-L)\cdot L = 3$, the image of $L$ is a singular
irreducible cubic $K$.

We have $f^{-1}(K) \in |3H-3L|$;
since
$(3H-4L)\cdot M_i =-1$
for $M_i\in\beta^*_1$, this cubic
$K$ passes through all
nodes $p_1,\ldots,p_6$ and
$$f^{-1}(K) = M_1+M_2+M_6+M_8+M_{9}+M_{12}+L+L'\in |3H-3L|,$$
where
$H\cdot L'=8$
and $L^{\prime2} = -2$.
Since $L$ is the image of the line $L_1$,  following the classical terminology, we say that
the node of $K$ is
\emph{apparent},
i.e., it is resolved under the
double cover. The cubic is tangent to the branch curve $B$ at three
smooth points
(which may collide),
so that $L$ intersects $L'$ at the pre-images of these three points and at
the two points corresponding to the branches at the node of $K$.

\begin{remark}\label{rem.6model}
As is well known, the condition that a double cover
$\pi\colon X\to \bbP^2$
branched along a nodal
sextic curve $B$  has a quartic birational model with an extra
node is the existence of a contact conic that passes through the nodes
of the sextic. The proper transform of the conic
under the double cover splits into the union of two curves $C_1+C_2$ and the linear system
$|\pi^*\calO_{\bbP^2}(1)+C_1|$ maps $X$ to a quartic surface blowing down $C_1$ to a node.

Along these lines, the condition that $X$ admits a birational sextic model with a line is the existence
of a contact cubic $K$ with an apparent node; this model is smooth if and only if $K$ passes through all nodes of~$B$.
The proper transform
of $K$ splits into
the union $L+L'$ of smooth rational curves
and the linear system $|\pi^*\calO_{\bbP^2}(1)+L|$ maps $X$ birationally onto a
surface of degree 6; the image of $L$ is a line. The ramification curve $\bar{B}$ of the cover belongs to
$|L+L'|$.
\end{remark}

\section{The Picard lattice}\label{s.NS}
Let $\Pic(X)$ be the Picard lattice of a general Humbert sextic K3 surface.
In this section, we
show that the 24 lines $L_i,M_j$ generate a primitive sublattice
$S$ of rank $15$. We give a $\bbZ$-basis of $S$ and
compute
the discriminant
quadratic form
of $S$, upon which we conclude that $S=\Pic(X)$.

\subsection{Hyperbolic bipartite graphs}\label{s.graphs}
Given a graph~$\Gamma$ with $n$ vertices, we denote by $\Z\Gamma$ the quadratic lattice
of rank $n$ with the Gram matrix $G  = A-2\mathbb{I}_n$, where $A$ is the adjacency matrix of
$\Gamma$ and $\mathbb{I}_n$ is the identity matrix
of size $n$. For a lattice~$L$, we let
$$\rad L: = L^\perp$$
be the radical of $L$, the kernel of the map $\Z\Gamma \to (\Z\Gamma)^\vee$ defined by the Gram matrix.
We often abbreviate $L/\!\rad L$ to $L/\!\rad$.

We take for $\Gamma$ the bipartite graph with
the bipartition $(\alpha,\beta)$ and
the adjacency
relation
defined by \autoref{table}.

Let
$$ S: = \Z\Gamma/\!\rad.$$
 With the geometric applications in mind, we refer to the vertices
of~$\Gamma$ as lines.
We reiterate that $S$ readily contains the ``$6$-polarization'' $H=\sum q$, where $q$
is any of the $(3,3)$-fragments (see \autoref{rem.3,3.property}).

\begin{proposition}\label{prop.basis}
The lattice $S$ is of rank 15 and freely generated by the lines
\beq\label{basis}
(L_2, L_3, L_4, L_5, L_6, L_7, L_8, L_{11}, M_1, M_2, M_3, M_6, M_8, M_9, M_{12}).
\eeq
 The discriminant group $S^\vee/S$ of
$S$ is isomorphic
to $(\Z/2\bbZ)^{\oplus 4}\oplus \Z/16\Z$.
\end{proposition}

\begin{proof} The Gram matrix $G$ of $\Z \Gamma$ can be written in the form
$$
\bmatrix-2\mathbb{I}_{12}&P\\
P&-2\mathbb{I}_{12}\endbmatrix,
$$
where $P$ is
the incidence matrix from \autoref{table}.  We can compute its integral Smith
normal form
to check the assertions about the rank and the discriminant group. We also compute the intersection matrix
of the sublattice spanned by the lines from \eqref{basis} and check that its
rank equals 15
and its discriminant group coincides with that of $S$.
\end{proof}

\begin{remark}\label{rem.radical}
It is easily seen that
$\rad(\Z\Gamma+\Z H)$ is generated, over~$\Z$,
by the classes of the form $H-\sum q$ (see \autoref{rem.3,3.property}),
where $q$ is one of the $16$ $(3,3)$ fragments in~$\Gamma$. Furthermore, one
can find a free basis for $\rad\Z\Gamma$
consisting of vectors of the form $\sum q'-\sum q''$.
\end{remark}

\begin{remark}\label{rem.others}
By brute force (starting with a discrete graph on $12$ vertices and adding
$12$ more vertices one-by-one), up to isomorphism there are but six bipartite
graphs~$\Gamma_i$ of type $(12_6,12_6)$  that define a hyperbolic lattice
$S_i$;
we let $S_1=S$ as above. It is remarkable that all these lattices are of rank 15,
admit a vector $h\in S_i\otimes \Q$ such that $h\cdot v = 1$ for each vertex of the graph,
and have $h^2=6$. If $S_i\ne S$,
the graph violates the $(3,3)$-property of \autoref{rem.3,3.property} and,
hence, cannot be realized as the full graph of lines on a smooth sextic K3
surface. For example, see $L_1,L_3,M_1,M_3,M_4$ in \autoref{fig.Gamma2}.

In fact, it can be shown that only the graphs $\Gamma_1 = \Gamma$ and $\Gamma_2$ defined by the
intersection matrix
in \autoref{fig.Gamma2}
can be realized as a part of the graph of lines on a K3 surface,
smooth or singular, of degree $6$ (if $i = 1$) or degree $4$ (if $i = 1,2$).
In all three cases,
a general surface
is smooth and $\Gamma_i$ is
its full graph of lines.
\begin{figure}[htb]
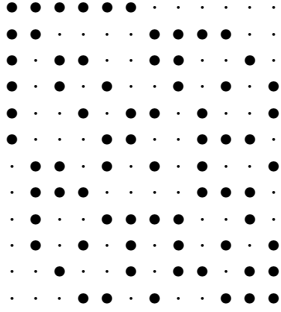

$$
\tconfig
 1 1 1 1 1 1 . . . . . . \cr
 1 1 . . . . 1 1 1 1 . . \cr
 1 . 1 1 . . 1 1 . . 1 . \cr
 1 . 1 . 1 . . 1 . 1 . 1 \cr
 1 . . 1 . 1 1 . 1 . . 1 \cr
 1 . . . 1 1 . . 1 1 1 . \cr
 . 1 1 . 1 . 1 . 1 . . 1 \cr
 . 1 1 1 . . . . 1 1 1 . \cr
 . 1 . . 1 1 1 1 . . 1 . \cr
 . 1 . 1 . 1 . 1 . 1 . 1 \cr
 . . 1 . . 1 . 1 1 . 1 1 \cr
 . . . 1 1 . 1 . . 1 1 1 \cr
\endtconfig
$$
\caption{The graph $\Gamma_2$}\label{fig.Gamma2}
\end{figure}
\end{remark}

\subsection{The Picard lattice of~$X$}\label{s.Picard}
We
have established that the Picard lattice $\Pic(X)$ of $X$
contains a sublattice $S$ of rank 15 spanned by the 24 lines.
The following proposition computes the discriminant form
$q\colon\discr(S)\to \bbQ/2\bbZ$ on the discriminant group
$\discr(S) = S^\vee\!/S \cong
(\bbZ/2\bbZ)^{\oplus 4}\oplus \bbZ/16\bbZ$.

\begin{proposition}\label{prop.discr}
The discriminant quadratic form $q$ on $\discr(S)$
is defined by the following Gram matrix:
$$
\bmatrix0&\frac12\\\frac12&0\endbmatrix\oplus
\bmatrix0&\frac12\\\frac12&0\endbmatrix\oplus\bmatrix{\frac{3}{16}}\endbmatrix,
$$
where, as usual, the diagonal entries are considered defined $\bmod\,2\Z$
whereas the others are defined $\bmod\,\Z$.

\end{proposition}

\begin{proof}
It is straightforward that, in the basis \eqref{basis}
introduced in \autoref{prop.basis},
the five vectors
$\gamma_1,\ldots,\gamma_5\in S\otimes\Q$ given by
$$
\setbox0=\hbox{$00$}
\vcenter{\openup2pt\halign{\hss$#=[\ $&&\hbox to\wd0{\hss$#$\hss}\,\ \cr
 2\gamma_1& 0& 0& 1& 1& 1& 1& 0& 0& 0& 0& 0& 0& 0& 0& 0&\omit$],$\hss\cr
 2\gamma_2& 0& 0& 0& 0& 0& 0& 0& 0& 1& 1& 0& 1& 1& 0& 0&\omit$],$\hss\cr
 2\gamma_3& 0& 0& 0& 0& 0& 0& 0& 1& 0& 1& 0& 0& 1& 1& 0&\omit$],$\hss\cr
 2\gamma_4& 1& 1& 1& 1& 0& 0& 0& 1& 1& 1& 1& 0& 0& 0& 0&\omit$],$\hss\cr
16\gamma_5& 2& 2& 4& 2& 10& 12& 2& 1& 1& 7& 1& 2& 12& 12& 10&\omit$]$\hss
\crcr}}
$$
belong to $S^\vee$. The Gram matrix of these vectors is as in the statement and, since
the matrix is nondegenerate, modulo $S$ they generate a group of the
correct size $256=\left|\discr(S)\right|$.
\end{proof}

\begin{theorem}\label{trans} Let  $X$ be a Humbert sextic K3 surface with
the Picard number $15$. Then,
$\Pic(X) \cong S$ and the transcendental lattice $T(X)$ is isomorphic to
$$
\def\-{\phantom-}
T\cong \bmatrix -2& \-1& \-0\\ \-1& -2& -1\\ \-0& -1& -6\endbmatrix
 \oplus\bmatrix0&2\\2&0\endbmatrix\oplus\bmatrix0&2\\2&0\endbmatrix.
$$
\end{theorem}

\begin{proof} By the assumption, $\Pic(X)$ contains $S$ as a sublattice of finite index.
The overlattices of $S$
of finite index correspond to isotropic subgroups of
$\discr(S)$. A non-trivial isotropic subgroup would contain an
element of order~$2$ and, up to the action of
the group $\Sym(\Gamma)$
(see \autoref{obs.group}), there are but two such
elements, namely, $\gamma_1$ and $8\gamma_5$. They are represented by the rational vectors
$$
\delta_1=\tfrac12(M_6-M_8+M_9-M_{12}),\quad
\delta_2=\tfrac12(L_5-L_6+L_7-L_8),
$$
with $\delta_i^2=-2$ and $\delta_i\cdot H=0$. If $\delta_i\in\Pic(X)$, then
$\pm\delta_i$ would be represented by a smooth rational curve contracted by
$|H|$
and $X$ would be singular.

It is immediate that $\discr(T) =-\discr(S)$ and the  signature $\sigma(T)=-3$,
i.e.,
$T$ is in the genus of $T(X)$.
Due to \cite[Theorem 1.14.2]{Nikulin}, this particular genus
consists of a single isomorphism class; hence, $T(X)\cong T$.
\end{proof}

\begin{corollary}\label{cor.uniqueness}
Let
$Y$ be a K3 surface and $\phi\colon\Z\Gamma/\!\rad\to\Pic(Y)$ an isomorphism.
Then $Y$ is isomorphic to a Humbert sextic~$X$ so that the classes of lines
in $\Pic(Y)$ are the images of the vertices of $\Gamma$.
\end{corollary}

\begin{proof}
It suffices to show that the moduli space of lattice $S$ polarized K3
surfaces is irreducible. By \cite[Proposition 5.6]{DolgachevK3}, the latter
follows from the fact that $T(Y)$ contains an admissible $2$-isotropic
vector.
\end{proof}

With \autoref{rem.others} taken into account, we have also the following
restatement.

\begin{corollary}\label{cor.uniqueness.sextic}
Let
$X\subset\Cp4$ be a smooth sextic K3 surface with a subgraph of type
$(12_6,12_6)$ in its configuration of lines.
Then $X$ is projectively isomorphic to a member of the Humbert family.
\end{corollary}

\begin{remark}\label{HumbertDesmic} The original construction of a Humbert sextic
surface depends on 5 parameters:
the choice of a hyperplane section of the
Humbert cubic complex $\mathfrak{G}$.
This agrees with the fact that the Picard number of a general Humbert
sextic equals $20-5=15$.

The sublattice $U(2)\oplus U(2)$
(the last two summands in~$T$)
contains a
primitive sublattice  $U(2)\oplus \la 4\ra$ isomorphic to the
transcendental lattice of a minimal resolution of the Kummer surface
of the self-product of an elliptic curve. As we know, the latter is birationally isomorphic to a Desmic quartic surface. Thus, we expect that the closure of the family of Humbert
sextic surfaces
contains a one-parameter family of sextic surfaces (probably singular) birational to desmic quartic surfaces. Unfortunately,
we were not able to find this family explicitly.
\end{remark}

\subsection{An alternative construction}\label{s.from.cubic}
A cubic hypersurface $V_3^n$ in $\bbP^n$ that
has an ordinary node $P= (1:0:\cdots:0)$ is given by an equation
$$x_0f_2(x_1,\ldots,x_n)+f_3(x_1,\ldots,x_n) = 0,$$
where $f_2, f_3$ are
some
homogeneous polynomials of degree 2, 3, respectively.
The variety  $\cone_P(V_3^n): =V(f_2,f_3)\subset\Cp{n-1}$
is called the \emph{associated variety} of
$(V_3^n,P)$;
it is a complete intersection of degrees $(2,3)$
in $\bbP^{n-1}$
parametrizing the lines in $V_3^n$ through $P$.
Conversely, every
degree~$(2,3)$ complete intersection $Y\subset\Cp{n-1}$ gives rise to
a cubic hypersurface
$V_3^n$ with a distinguished double point ~$P$: it is
the image of
$\bbP^{n-1}$ under the rational map $\bbP^{n-1}\dasharrow \bbP^n$
given by the
linear system of cubic hypersurfaces in $\bbP^{n-1}$ containing
$Y$. If $Y$ has but $N$ nodes as
singularities and the quadric $V(f_2)$ is nonsingular, then
 $\cubic(Y)$
has $N+1$ nodes, the extra node~$P$ being
the image of the exceptional divisor of the blow-up of
$Y$ in $\bbP^{n-1}$.

It was shown in \cite{DolgKondo} that the Humbert cubic complex is projectively isomorphic to the
associated variety of the Segre cubic hypersurface in $\bbP^6$ given by the equation
$$\sum_{i=0}^6x_i^3-\left(\sum_{i=0}^6x_i\!\right)^{\!\!3} = 0.$$
Thus, alternatively, we can define the Humbert sextic K3 surface as the transversal intersection
of the associated variety of the Segre cubic hypersurface in $\bbP^6$.

\section{Automorphisms of $X$}\label{s.aut}

In this section, we discuss the automorphism of general and some special
representatives of the Humbert family.

\subsection{Projective automorphisms}\label{s.aut.projective}
We start with asserting that a general Humbert sextic $X$ has no
automorphisms induced from $\operatorname{PGL}(5,\C)$.

\begin{lemma}\label{lem.extension}
Any projective automorphism $\sigma$ of a Humbert sextic K3 surface $X$
extends to a projective automorphism of the Humbert line complex $\frakG$.
\end{lemma}

\begin{proof}
Recall that $\frakG$ is
cut off the quadric $\quadric$
by a cubic hypersurface.
Assume that $X=\frakG\cap\plane$ for a hyperplane~$\plane$.
Then $X$ lies in $\quadric_0=\quadric\cap\plane$ and, hence, we have
$\sigma(\quadric_0)=\quadric_0$, as
otherwise $X$ would be contained in the quartic surface
$\quadric_0\cap\sigma(\quadric_0)$.
Letting $\quadric=V(q)$ for a quadratic form~$q$, we conclude that
(an appropriately scaled lift to $\C^5$ of) $\sigma$ is an
automorphism of the restriction $q|_\plane$. By Witt's extension theorem, it extends to an
automorphism $\tilde{\sigma}$ of~$q$, so that
$\tilde{\sigma}(\quadric)=\quadric$.

By \autoref{rem.degenerations} below, $\sigma$ preserves as a set the
``original'' $24$ lines $L_1,\ldots,M_{12}$ on~$X$. Recall that each line in
the Grassmannian
$\quadric=G_1(\Cp3)$ is contained in exactly one $\alpha$-plane and exactly one
$\beta$-plane. For $N=L_1,\ldots,M_{12}$, denote by $\pi^+_N$ (resp.\
$\pi^-_N$) the plane of the same (resp.\ opposite) type $\alpha$ or~$\beta$
as~$N$ (as defined in \autoref{conv.alpha-beta}).
In other words, $\pi^+_N$ are the $24$ planes~\eqref{eq.alpha-beta}
whereas $\pi^-_N$ are some ``wrong'' planes most likely not
even contained in~$\frakG$.
Since $\sigma$ respects (i.e., simultaneously preserves or simultaneously reverses)
the type of lines (see \autoref{obs.group})
and $\tilde{\sigma}$ respects the type of planes, we have
$\tilde\sigma(\pi^+_N)=\pi^\epsilon_N$ for a
\emph{constant} $\epsilon=\epsilon(\tilde\sigma)=\pm1$.
We consider separately the following two cases.

\smallskip
\underline{Case 1}:
the restriction $q|_\plane$ is degenerate (necessarily of corank~$1$).

\smallskip
Since
the line $N$ lies in the cone $\quadric_0$ \emph{not passing through its
vertex},
exactly one of $\pi^\pm_N$ lies in
$\quadric_0\subset\plane$ and, by \autoref{rem.tangent}, it is $\pi^-_N$.
This property clearly distinguishes the two planes and we have
$\tilde\sigma(\pi^+_N)=\pi^+_N$ for each $N=L_1,\ldots,M_{12}$.
Therefore,
$\frakG\cap \tilde{\sigma}(\frakG)$ contains
the $24$ planes~\eqref{eq.alpha-beta} and
we conclude that $\tilde{\sigma}(\frakG)=\frakG$:
indeed,
otherwise $\frakG\cap \tilde{\sigma}(\frakG)$ would be a surface of
degree~$18$.

\smallskip
\underline{Case 2}:
the restriction $q|_\plane$ is non-degenerate.

\smallskip
This time, there are two extensions~$\tilde\sigma$: they differ by a
``reflection'' against~$\plane$ (a choice of sign $\pm1$ in the direction
$q$-orthogonal to $\plane$). Since the reflection itself is type reversing, a
unique
extension~$\tilde\sigma$ can be \emph{chosen}
so that $\tilde\sigma(\pi^+_N)=\pi^+_N$, upon
which the proof concludes in the same manner as in the previous case.
\end{proof}

\begin{theorem}
\label{prop.projective}
The subgroup of $\Aut(\bbP^4)$ that leaves a general
Humbert sextic $X$ invariant is trivial.
\end{theorem}

\begin{proof}
By \autoref{lem.extension},
the group of projective automorphisms
of a Humbert sextic K3 surface $X = \frakG\cap\plane$ coincides with the group $\Aut_\plane(\frakG)$
of projective automorphisms of $\frakG$ leaving the hyperplane $\plane$ invariant.
Since $\Aut(\frakG)$ is finite (see~\cite{DolgKondo} and
\autoref{s.eqns}), the points $\plane\in\check{\bbP}^5$ invariant under the
natural action of $\Aut(\frakG)$ or a non-trivial subgroup thereof constitute a proper
Zariski closed set.
\end{proof}

\begin{remark}\label{rem.discr.action}
In general, if $\rk S<20$, the group of projective automorphisms of a very
general lattice $S$-polarized K3 surface~$X$ is computed as the pull-back
$\rho^{-1}(\pm\id)$ under the natural homomorphism
$$
\rho\colon O_H(S) \to O(\discr S),
$$
where $O_H(S)\subset O(S)$ is the (finite) subgroup preserving the
polarization $H\in S$.
(If, as in \autoref{prop.projective}, $\rk S$ is odd, the somewhat vague
``very general'' can be replaced with the requirement $\Pic(X)=S$.)
Since, in our case, $S$ is generated by lines,
$O_H(S)=\Sym\Gamma$, whereupon, using \autoref{prop.discr} (and the proof thereof),
\autoref{obs.group}, and \texttt{GAP}~\cite{GAP4.13},
we arrive at $\rho^{-1}(\pm\id)=\{\id\}$.
\end{remark}

\subsection{Birational automorphisms}\label{s.aut.abstract}

In spite of \autoref{prop.projective},
we claim that
the full group $\Aut(X)$ of automorphisms of $X$ is infinite.
For example,
in
\autoref{s.elliptic} below we find quite a few
elliptic fibrations on $X$. Some
admit a section
$E$ so that the
divisor classes of $E$ and of the irreducible components of the fibers span a
positive corank sublattice of
$\Pic(X)$.
By
the Shioda--Tate formula \cite[Chapter 11, Corollary 3.4]{Huybrechts},
the Mordell--Weyl group of translation automorphisms of $X$
along the fibers of the elliptic fibration is infinite
(e.g., see \autoref{rem.MW} below).

Another
approach would be using
the classification of the Picard lattices
of the algebraic K3 surfaces with finite automorphism group
found in~\cite{Nikulin2}:
our lattice $S=\Pic(X)$ is not on the list.

Besides,
we have 24 involutions
$\tau_N\colon X\to X$ each of which is the
covering transformation of the projection $X\to \bbP^2$
from a line $N$ on $X$ (see \autoref{planemodel}).
For any pair of lines, the two covering involutions generate an infinite dihedral group.
We do not know whether these 24 involutions generate the whole group
$\Aut(X)$.

\begin{remark}
In fact,
any smooth sextic K3 surface with at least two lines has infinite group of
birational automorphisms. Furthermore, it can be shown that,
with very few exceptions, the
involutions defined by a pair of distinct lines generate an infinite dihedral
group.
Proofs and details will appear in~\cite{degt.Rams:sextics}.
\end{remark}

\subsection{Anti-symplectic involutions and cubic surfaces}\label{s.aut.example}
In principle, \autoref{lem.extension} and the description of $\Aut(\frakG)$
found in~\cite{DolgKondo} (see \autoref{s.eqns}) let us find all Humbert
sextic K3 surfaces admitting a non-trivial projective automorphism.
Below we confine ourselves to a maximal stratum with an
anti-symplectic involution. Other examples are mentioned in
\autoref{rem.degenerations} below.


\begin{example}\label{cubics}
Let $a_1 = a_2 = 1$ in \eqref{eqCC},~\eqref{eq.HS}.
Then, $X$ admits an anti-symplectic involution
$\gamma\colon x_1\leftrightarrow x_2$
with $X^\gamma$
a smooth hyperplane section $C = V(x_1-x_2)$. Let $s = x_1x_2$, $t = x_1+x_2$. Equations
\eqref{eqCC},~\eqref{eq.HS} can be rewritten in the form
\beq\label{EqCLC}
\begin{gathered}
t^2-2s+x_3^2+y_1^2+y_2^2+y_3^2 = 0,\\
sx_3+\ii y_1y_2y_3 = 0,\\
t+a_3x_3+a_4y_1+a_5y_2+a_6y_3 = 0.
\end{gathered}
\eeq
These equations define a surface in $\bbP(1^4,2)$, where we weight $s$ with degree $2$ and other coordinates with degree $1$.
Projecting from the point $(0:0:0:0:1)$,  we obtain that $Y$ is isomorphic to the cubic surface
in $\bbP^3$ given by an equation
$$
 F = \bigl((a_3x_3 + a_4y_1 + a_5y_2 + a_6y_3)^2 + x_3^2 + y_1^2 + y_2^2 + y_3^2\bigr)x_3 + 2\ii y_1y_2y_3.
$$

The
plane $V(x_3)$ is a tritangent
plane of $Y$. Recall that any line $\ell$ on the cubic surface $Y$ is contained in one of the plane sections of $Y$ containing
one of the lines $V(x_3,y_i)$. Plugging in $x_3= ky_i$ in the equation of the cubic surface, we
find that the residual conic in the plane $V(x_3-ky_i)$ is given by a symmetric $3\times 3$ matrix
 whose entries are homogeneous polynomials in $k$ (of different degrees).

  Computing the determinant of this matrix, we find that
 the parameters $k$ defining singular residual conics are the zeros of the polynomial
\beq\label{parameterk}
k(k^2+1)\bigl((a_3^2 + a_5^2+ a_6^2+1)k^2-2(\ii a_5a_6-a_4a_3)k+a_4^2+1\bigr)= 0.
\eeq
The parameter $k = 0$ corresponds to the tritangent plane $V(x_3)$ containing the residual conic equal to the union
of the two lines $\ell_j\ne \ell_1$. We choose the parameters
$k = \pm \ii$. The branch curve $B$ of the cover $X\to Y$ is cut out by the quadric
given by the equation
$$t^2-4s = (a_3x_3 + a_4y_1 + a_5y_2 + a_6y_3)^2+2(x_3^2+y_1^2+y_2^2+y_3^2) = 0.$$
 Plugging in the equation
$x_3=  \ii y_1$, we find
$$(a_3x_3 + a_4y_1 + a_5y_2 + a_6y_3)^2+2(y_2^2+y_3^2) = 0,$$
and plugging in $F = 0$, we get
$$(a_3x_3 + a_4y_1 + a_5y_2 + a_6y_3)^2+y_2^2+y_3^2+2y_2y_3 = 0.$$
Together this gives $(y_2+y_3)^2 = 0$.  This shows that each line on $Y$ contained in the plane
$V(x_3-\ii y_1)$ is tangent to the branch curve $B$. Similarly, we find that all $12$ lines in
$Y$ contained in the planes $V(x_3-\ii y_i)$ are tangent to $B$.
\end{example}

\begin{remark} The set of 24 lines in $Y$ lying in the planes $V(x_3\pm ky_i)$, where $k\ne 0$ satifies
equation \eqref{parameterk} is equal to the union of two sets of 12 lines, one of which is a double-six of lines.
By checking the intersections of the 12  lines corresponding to the parameters $k = \pm \ii$, we ontain that our
set is complementary to the double-six.
\end{remark}

\begin{remark}\label{rem.any.cubic?}
Let
$\mathcal{M}$ be the moduli space of pairs $(F,C)$, where $F$ is a smooth cubic surface and $C$ is a smooth curve
in $|{-2K_F}|$. It is a variety of dimension $13$, a projective bundle over the moduli space of smooth cubic surfaces.
Taking the double cover of $F$ branched along $C$, we obtain a finite map to the moduli space $\mathcal{M}'$ of K3 surfaces polarized by the lattice
$\mathsf{I}^{1,6}(2)\cong \Pic(F)(2)$. We find it amazing that our $5$-dimensional moduli space
of Humbert sextic K3 surfaces intersects $\mathcal{M}'$ along a four-dimensional subvariety.
We believe, but could not prove it,  that
the pre-image of this subvariety in $\mathcal{M}$ is of dimension $4$ and projected surjectively onto the moduli space of cubic surfaces.
\end{remark}

\subsection{Anti-symplectic involutions via lattices}\label{s.aut.as}
Constructed
in \autoref{cubics} is a $4$-parameter family of Humbert
sextics~$X$, each having $12$ extra conics. It follows that, for $X$ general,
the corank~$1$ space $\Pic(X)\otimes\Q$
is generated over~$S\otimes\Q$ by any of these conics.
Each conic can be regarded as a vector
$c\in\Hom(\Z\Gamma+\Z H,\Z)$; as such,
our ``symmetric'' conics have the following properties:
\begin{itemize}
\item
$c(H)=2$ and $c(N)=0$ or~$1$ for each line, i.e., vertex~$N$ of $\Gamma$;
\item
$c(N)=1$ for exactly four $\alpha$-lines and exactly four $\beta$-lines;
\item
$c$ annihilates $\rad(\Z\Gamma+\Z H)$, cf. \autoref{rem.radical}.
\end{itemize}
Apart from the nine old conics, there is a single $G$-orbit of such
vectors~$c$, and we take for one of the new conics~$C$ the one intersecting
$$
\DEFS \l5, \l8, \l10, \l12, \l18, \l19, \l21, \l23,
$$
see the first row in \autoref{fig.conics} below. (We change the indexing
from~\autoref{s.aut.example} for
a better looking matrix.) Then, we take
\beq\label{newlattice}
\bar{S}=(\Z\Gamma+\Z H+\Z C)/\!\rad
\eeq
for the new Picard lattice.
A straightforward computation in the spirit of \autoref{s.NS} and
\autoref{s.rational} below shows that
\begin{itemize}
\item
$\bar{S}$ is indeed the Picard lattice of a smooth sextic K3 surface X;
\item
no non-trivial finite index extension of $\bar{S}$ has this property; hence,
the new family constructed is indeed the one in \autoref{s.aut.example};
\item
$X$ has the $24$ old lines, $9$ old conics, and $12$ new conics.
\end{itemize}
The new conics
are depicted in \autoref{fig.conics}. The divisor classes of
six of them are shown in the figure, and others are found recursively:
given a new conic $C'$, four more are $H-C'-L_i-M_j$, where
$L_i\cdot C=M_j\cdot C=L_i\cdot M_j=1$. (A more conceptual explanation of this phenomenon is
found further in this section, where some of the conics on $\bX$
are interpreted as lines on a
cubic surface.)
Note that the
new conics appear in pairs indistinguishable by their intersection with the
lines.
They are ordered so that $\{\text{odd}\}$, $\{\text{even}\}$ constitute a
 double-six, see the bottom right corner of \autoref{fig.cubic}
below for the intersection matrix (divided by $2$).

\begin{figure}[htb]
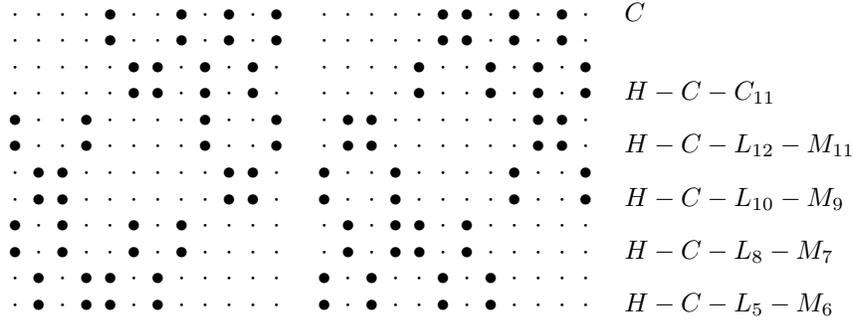

$$
\def\\#1{&\omit\small\quad\smash{$\DEFS#1$}\hss\cr}
\let\1=1
\tconfig
 . . . . 1 . . 1 . 1 . 1 & . . . . . 1 1 . 1 . 1 . \\{C}
 . . . . 1 . . 1 . 1 . 1 & . . . . . 1 1 . 1 . 1 . \cr
 . . . . . 1 1 . 1 . 1 . & . . . . 1 . . 1 . 1 . 1 \cr
 . . . . . 1 1 . 1 . 1 . & . . . . 1 . . 1 . 1 . 1 \\{H-C-C_{\1\1}}
 1 . . 1 . . . . 1 . . 1 & . 1 1 . . . . . . 1 1 . \cr
 1 . . 1 . . . . 1 . . 1 & . 1 1 . . . . . . 1 1 . \\{H-C-L_{\12}-\l23}
 . 1 1 . . . . . . 1 1 . & 1 . . 1 . . . . 1 . . 1 \cr
 . 1 1 . . . . . . 1 1 . & 1 . . 1 . . . . 1 . . 1 \\{H-C-L_{\10}-M_9}
 1 . 1 . . 1 . 1 . . . . & . 1 . 1 1 . 1 . . . . . \cr
 1 . 1 . . 1 . 1 . . . . & . 1 . 1 1 . 1 . . . . . \\{H-C-\l8-M_7}
 . 1 . 1 1 . 1 . . . . . & 1 . 1 . . 1 . 1 . . . . \cr
 . 1 . 1 1 . 1 . . . . . & 1 . 1 . . 1 . 1 . . . . \\{H-C-\l5-M_6}
\endtconfig
$$
\caption{The twelve new conics}\label{fig.conics}
\end{figure}

\begin{remark}\label{rem.conic.families}
There are several $4$-parameter families of Humbert sextics with the same set
of lines and a few extra conics,
no longer symmetric (cf.\ also \autoref{rem.degenerations} below
concerning extra lines). However, we only consider the one described.
\end{remark}

Let $\bGamma$ be the colored graph of lines and conics on~$\bX$. Using the
\texttt{digraph} package in \texttt{GAP}~\cite{GAP4.13}, we can compute the
group $\bG=\Sym(\bGamma)$. We have $|\bG|=192$; this group acts transitively
on the set of lines and on that of the new conics, and the action of $\bG$ on
the set of old conics has two orbits, one being $\{C_{11},C_{22},C_{33}\}$.
The action of $\bG$ on the lines (which almost determines its action on
the conics) is the set-wise stabilizer of the collection
(shown in \autoref{fig.conics})
$$
\bigl\{\{N\in\Gamma\,|\,N\cdot C=1\}\bigm|\text{$C$ is a new conic}\bigr\}
$$
under $G=\Sym(\Gamma)$, see \autoref{obs.group}. The kernel of this action is
generated by an involution interchanging the two Schl\"{a}fli double-sixes, see below.

Next, we argue as explained in \autoref{rem.discr.action} and conclude that
$\Aut_H\bX=\Z/2$.
The generator $\gamma$ of this group is an anti-symplectic
projective involution of~$X$:
it acts on the lines and old conics via
$$
L_{2i}\leftrightarrow M_{2i-1},\quad
L_{2i-1}\leftrightarrow M_{2i},\quad
C_{rs}\leftrightarrow C_{sr},\ r\ne s,
$$
leaving invariant $C_{rr}$, $r=1,2,3$, and all new conics.

\begin{remark}\label{rem.other.models}
A posteriori
we can consider the involution
$\gamma_T=-r_v\colon T\to T$ (see \autoref{prop.discr}), where $r_v$ is the
reflection defined by an appropriate square $(-4)$ vector $v\in T\cap2T\dual$
(e.g., the difference of the two generators of one of the two $U(2)$-summands).
Then we take $\bT=v^\perp$ for the new transcendental lattice, so
that the new Picard lattice is an index~$2$ extension
$\bS\supset S\oplus\Z v$. In more detail this approach will be discussed
 by the first author elsewhere.

Yet another ``systematic'' way would be to analyze, in the spirit of
\autoref{rem.others}, the realizations of $\Gamma$ as the graph of lines on a
$2$-polarized K3 surface. This time we would have to allow singularities,
arriving at a single $4$-parameter family of $6$-nodal double planes.
As explained in the next paragraph,
this model is merely the cubic surface $\bX/\gamma$ with
a sextuple of disjoint lines coming from invariant conics on~$\bX$
contracted.
\end{remark}

\begin{figure}[htb]
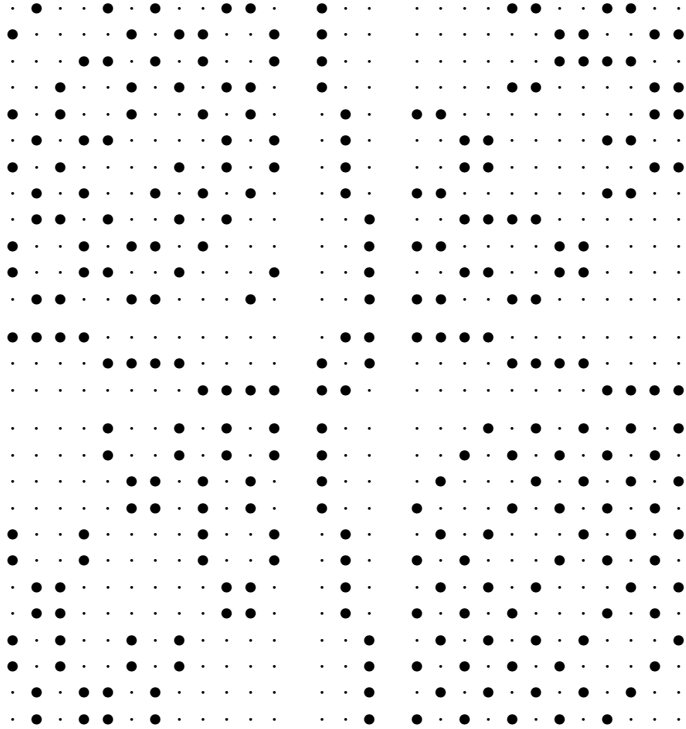

$$
\def\\{\noalign{\medskip}}
\tconfig
 . 1 . . 1 . 1 . . 1 1 . & 1 . . & . . . . 1 1 . . 1 1 . . \cr
 1 . . . . 1 . 1 1 . . 1 & 1 . . & . . . . . . 1 1 . . 1 1 \cr
 . . . 1 1 . 1 . 1 . . 1 & 1 . . & . . . . . . 1 1 1 1 . . \cr
 . . 1 . . 1 . 1 . 1 1 . & 1 . . & . . . . 1 1 . . . . 1 1 \cr
 1 . 1 . . 1 . . 1 . 1 . & . 1 . & 1 1 . . . . . . . . 1 1 \cr
 . 1 . 1 1 . . . . 1 . 1 & . 1 . & . . 1 1 . . . . 1 1 . . \cr
 1 . 1 . . . . 1 . 1 . 1 & . 1 . & . . 1 1 . . . . . . 1 1 \cr
 . 1 . 1 . . 1 . 1 . 1 . & . 1 . & 1 1 . . . . . . 1 1 . . \cr
 . 1 1 . 1 . . 1 . 1 . . & . . 1 & . . 1 1 1 1 . . . . . . \cr
 1 . . 1 . 1 1 . 1 . . . & . . 1 & 1 1 . . . . 1 1 . . . . \cr
 1 . . 1 1 . . 1 . . . 1 & . . 1 & . . 1 1 . . 1 1 . . . . \cr
 . 1 1 . . 1 1 . . . 1 . & . . 1 & 1 1 . . 1 1 . . . . . . \cr\\
 1 1 1 1 . . . . . . . . & . 1 1 & 1 1 1 1 . . . . . . . . \cr
 . . . . 1 1 1 1 . . . . & 1 . 1 & . . . . 1 1 1 1 . . . . \cr
 . . . . . . . . 1 1 1 1 & 1 1 . & . . . . . . . . 1 1 1 1 \cr\\
 . . . . 1 . . 1 . 1 . 1 & 1 . . & . . . 1 . 1 . 1 . 1 . 1 \cr
 . . . . 1 . . 1 . 1 . 1 & 1 . . & . . 1 . 1 . 1 . 1 . 1 . \cr
 . . . . . 1 1 . 1 . 1 . & 1 . . & . 1 . . . 1 . 1 . 1 . 1 \cr
 . . . . . 1 1 . 1 . 1 . & 1 . . & 1 . . . 1 . 1 . 1 . 1 . \cr
 1 . . 1 . . . . 1 . . 1 & . 1 . & . 1 . 1 . . . 1 . 1 . 1 \cr
 1 . . 1 . . . . 1 . . 1 & . 1 . & 1 . 1 . . . 1 . 1 . 1 . \cr
 . 1 1 . . . . . . 1 1 . & . 1 . & . 1 . 1 . 1 . . . 1 . 1 \cr
 . 1 1 . . . . . . 1 1 . & . 1 . & 1 . 1 . 1 . . . 1 . 1 . \cr
 1 . 1 . . 1 . 1 . . . . & . . 1 & . 1 . 1 . 1 . 1 . . . 1 \cr
 1 . 1 . . 1 . 1 . . . . & . . 1 & 1 . 1 . 1 . 1 . . . 1 . \cr
 . 1 . 1 1 . 1 . . . . . & . . 1 & . 1 . 1 . 1 . 1 . 1 . . \cr
 . 1 . 1 1 . 1 . . . . . & . . 1 & 1 . 1 . 1 . 1 . 1 . . . \cr
\endtconfig
$$
\caption{The $27$ invariant divisors}\label{fig.cubic}
\end{figure}
Thus, we have twelve invariant pairs
(split conics) $L_i+\gamma(L_i)$, $i=1,\ldots,12$,
three invariant old conics $C_{rr}$, $r=1,2,3$, and twelve invariant new
conics. The intersection matrix of these $12+3+12=27$ invariant divisors
obtained
(upon division by~$2$) is shown in \autoref{fig.cubic}, and one can readily
recognize the $27$ lines on the smooth cubic surface
$F=\bX/\gamma$.
The three lines in the middle lie in a tritangent plane, and the two
sextuples $\{2i\}$, $\{2i+1\}$, $i=8,\ldots,13$ (the even/odd ones of the last twelve)
constitute a Schl\"{a}fli double-six.
Any of these two sextuples can be blown down to obtain the double plane model
of \autoref{rem.other.models}.

\section{Back to the double plane model}\label{s.plain}
As we observed before, the family of Humbert sextic K3 surfaces depends on 5 parameters. However, their
double plane model seems
to require
$18+1+3+6 = 28$
conditions on plane sextics (depending on $27$ parameters) to obtain a
family of sextics with  6 nodes,
admitting a contact split nodal cubic,
six contact lines, and tangent to the
diagonals of the complete quadrilateral.
 Here, we solve this puzzle by proving that the splitting of the three diagonals and one line is
almost
enough
to obtain the double cover isomorphic to a Humbert sextic K3 surface.

We start with a complete quadrilateral
$\ell_{236}$, $\ell_{135}$, $\ell_{245}$, $\ell_{146}$ and a sextic curve~$B$ with
nodes at the six vertices $p_1,\ldots,p_6$. First, we require that each diagonal
$\ell_{12}, \ell_{13},\ell_{23}$
should be tangent to $B$ at some point
$q_i$
distinct from all points $p_i$.
A straightforward computation in the spirit of \autoref{s.NS} shows that there are two
families of such sextics, which differ by the proper transforms of the
diagonals: either
\begin{enumerate}
\item\label{fB.good}
the triangle of the diagonals lifts to a hexagon, as in
\autoref{fig.hex}, left, or
\item\label{fB.bad}
the triangle of the diagonals splits into two triangles, as in
\autoref{fig.hex}, right.
\end{enumerate}

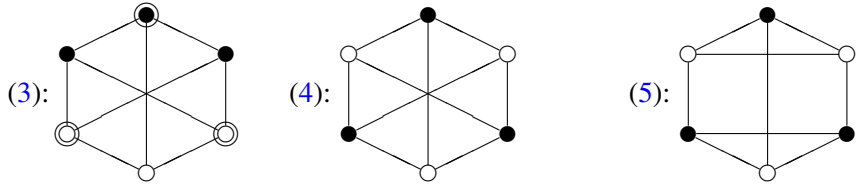
\begin{figure}
\unitlength1.5pt
\def\w{\circle{4}}
\def\ww{\circle{6}}
\def\b{\circle*{4}}
\def\hhex{%
\put(2,9){\line(2,-1){16}}
\put(38,9){\line(-2,-1){16}}
\put(0,12){\line(0,1){16}}
\put(40,12){\line(0,1){16}}
\put(2,31){\line(2,1){16}}
\put(38,31){\line(-2,1){16}}
\put(20,2){\line(0,1){36}}
}
\def\hex{%
\hhex
\put(2,11){\line(2,1){36}}
\put(38,11){\line(-2,1){36}}
}
\def\lb#1{\put(-15,18){\eqref{#1}:}}
\begin{picture}(40,40)(0,0)
\hex
\put(20,0)\w
\put(0,10)\w
\put(40,10)\w
\put(20,40)\b
\put(0,30)\b
\put(40,30)\b
\put(0,10)\ww
\put(40,10)\ww
\put(20,40)\ww
\lb{fBl.6}
\end{picture}
\qquad\qquad
\begin{picture}(40,40)(0,0)
\hex
\put(20,0)\w
\put(0,10)\b
\put(40,10)\b
\put(20,40)\b
\put(0,30)\w
\put(40,30)\w
\lb{fBl.2}
\end{picture}
\qquad\qquad\qquad
\begin{picture}(40,40)(0,0)
\hhex
\put(2,10){\line(2,0){36}}
\put(2,30){\line(2,0){36}}
\put(20,0)\w
\put(0,10)\b
\put(40,10)\b
\put(20,40)\b
\put(0,30)\w
\put(40,30)\w
\lb{fBl.4}
\end{picture}
\caption{The intersection of $M_3'$ with the pre-images of the diagonals}\label{fig.hex}
\end{figure}

\begin{remark}\label{rem.resolution}
The two families can be described geometrically.
Let $f\colon Y\to \bbP^2$ be the minimal resolution of singularities
of the double cover branched along $B$.
We factor $f$ as $ Y\to F'\to \bbP^2$,
where $F'$ is the weak del Pezzo surface of degree $3$ obtained by blowing up
the points
$p_1,\ldots,p_6$. The anti-canonical model of $F'$ is obtained by blowing down the proper transforms of
the diagonals $\ell_{ij}$. It is isomorphic to the $4$-nodal cubic surface $F$.
The proper transforms
of the diagonals can be identified with the sum $T$ of three lines $m_1,m_2,m_3$ cut out by a tritangent plane $\Pi$ of $F$. The proper transform $B'$ of $B$
is mapped to the intersection of $F$ with a quadric surface not passing through the nodes of $F$.
Counting constants, we find that there are two families
of quadrics, both of dimension $6$, intersecting
each line $m_i$ at one point with multiplicity two. In each family, one is required to impose three conditions on quadrics.

One family is defined by the condition
that a quadric $Q$ intersects the tritangent plane along a conic tangent to the lines $m_i$.
The other
family is defined by
the condition that $Q$ is tangent to the tritangent plane along a line $\ell$.
In the latter case, the pencil of planes containing $\ell$ defines a pencil of cubic curves on $F$. It is equal to the proper
transform of the pencil of cubic curves in the plane with base points $p_1,\ldots,p_6$ and the points of tangency
of $B$ with the diagonals.

The triangle $T$ of lines on the cubic surface $F$ is a reducible curve of arithmetic genus one.
Its pre-image in the double cover is a double cover branched along the Cartier divisor
$D = 2(q_1+q_2+q_3)\in |\calO_T(2)|$, where
$Q\cap T = \{q_1,q_2,q_c\}$.
Thus, we have two cases, resulting in two families of sextic curves~$B$.

\smallskip
\underline{Case \eqref{fB.good}}:
$Q\cap \Pi$ is a smooth conic tangent to $T$ at $q_1,q_2,q_3$.

\smallskip
The double cover is not trivial, i.e., it is not equal to the union of two curves mapped isomorphically to
$T$ under the covering. This corresponds to the first two pictures in \autoref{fig.hex}.
Its restriction over the open curve $T\setminus \{q_1,q_2,q_3\}$
corresponds to a subgroup of index two of its fundamental group.

The pre-image of $T$ is a reducible curve with the dual graph shown in
\autoref{fig.hex}, left.
Here and below, the action of the deck translation is the central
symmetry.

\smallskip
\underline{Case \eqref{fB.bad}}:
$Q\cap \Pi$ is a double line.

\smallskip
The cover is trivial, i.e., it defines the trivial cover of
 $T\setminus \{q_1,q_2,q_c\}$. The dual graph of the pre-image of $T$ is given in
 \autoref{fig.hex}, right.
 \end{remark}

Since we need to be able to choose pairwise disjoint pull-backs of the
diagonals, we concentrate on Case~\eqref{fB.good}, where
the triangle of the diagonas
does not split under the cover, i.e.,
the proper transform of the union of the diagonals is a hexagon (with the
three long diagonals, cf.\ \autoref{fig.hex}, left)
of $(-2)$-curves on $Y$.

Let
$M_1', M_2', M_6', M_8'$, $M_9', M_{12}' \subset Y$
be the proper transforms
of the
exceptional curves of the blow-up,
choose three disjoint sides $L_2', L_3', L_4'$
of the hexagon
(e.g., the circled vertices in \autoref{fig.hex}, left), and
let $L_5',L'_{11};L_6',L'_9; L_7',L'_{10}; L_8',L'_{12}$ be the
$(-2)$-curves that are mapped, in pairs, by
 the double cover $f\colon Y\to\Cp2$ to
the four sides of the
quadrilateral. They are all $(-2)$-curves on $Y$.

Now, we invoke one more condition that there exists a
tritangent
line $\ell$ of $B$ not passing through its nodes.
The
tritangent $\ell$ splits under the double cover into the union
of two $(-2)$-curves intersecting at three points;
denote them by $M_3'$, $M_3''$.
It can be shown that there are two irreducible families of pairs $(B,\ell)$, both
depending on $5$ parameters: either
\begin{enumerate}
\setcounter{enumi}2
  \item \label{fBl.6}
  $B$ is as in Case~\eqref{fB.good} and has six tritangents, or
  \item \label{fBl.2}
  $B$ is as in Case~\eqref{fB.good} and has two tritangents.
\end{enumerate}
In terms of $\ell$ itself only, the two families
differ by the intersection of $M_3'$ with the sides of the hexagon,
see the black vertices in \autoref{fig.hex}, left and center
(whereas the other pull-back $M_3''$ is represented
by the white vertices); in Case~\eqref{fBl.6}, the other pairs are
obtained by rotation.

Thus, we have
fifteen $(-2)$-curves
$$
M_1', M_2',  M_6', M_8', M_9', M_{12}', L_5', L_6',L_7',L_8',
L_2',L_3',L_4',L_{11}', M_3'.
$$
Here, we choose one full pair $L_5',L_{11}'$ while keeping only one
chosen line from the other three pairs.
Now, once the pairwise intersections of $M'_3,L_2',L_3',L_4'$ have been
arranged,
it is immediate to check that,
\emph{under the appropriate choice of
the components $(L_5',L'_{11}), L_6',L_7',L_8'$},
the intersection matrix of these curves coincides with
that of the curves constituting the basis~\eqref{basis}
on a general Humbert sextic K3 surface $X$. (Indeed, since the pull-backs of
the four sides of the quadrilateral
are disjoint from each other and from those of the diagonals,
we merely index them according to \autoref{table}.)
Therefore, these curves
span a lattice isomorphic to
$\Pic(X)$,
and it remains to apply \autoref{cor.uniqueness}.

\begin{remark}\label{rem.contact.cubics}
For completeness, in Case~\eqref{fB.bad}
there is a single $5$-parameter family of pairs
$(B,\ell)$:
\begin{enumerate}
\setcounter{enumi}4
  \item \label{fBl.4}
  $B$ is as in Case \eqref{fB.bad} and has four tritangents.
\end{enumerate}
A pull-back~$M_3'$ of~$\ell$ intersects those of the diagonals
as shown in \autoref{fig.hex}, right.

Whereas the existence of~\eqref{fBl.6} is guaranteed by
\autoref{prop.basis}, the existence of~\eqref{fBl.2} and~\eqref{fBl.4}, as
well as the very fact that there are but three families needs proof, which
will appear elsewhere.
It is also worth mentioning that, in all five cases
\eqref{fB.good}--\eqref{fBl.4}, i.e.,
$B$ itself or a pair
$(B,\ell)$,
a general curve~$B$ has quite a few contact cubics
with an apparent node. However, only in Case~\eqref{fBl.6} there is a (unique)
such cubic passing through all six nodes of~$B$, so that the corresponding
sextic K3 surface $X\subset\Cp4$ is smooth (see \autoref{rem.6model}).
\end{remark}


\begin{remark}\label{rem.deg=4}
One can compare the specialty of the plane sextic $B$ with respect to the tangency conditions with
another plane sextic curve, known as the \emph{Humbert plane sextic} of genus 5. It has five cusps
and is tangent to any
line connecting
a pair of cusps as well as the unique conic passing through the
cusps.
The double cover of $\bbP^2$ branched along the Humbert sextic is birationally isomorphic to the Kummer quartic surface
associated with a nonsingular curve of genus 2 \cite[Remark 8.6.9]{CAG}. As in our case, the 16 tangency condition would wrongly
imply that such a curve
should not exist.

\end{remark}

\section{Elliptic pencils}\label{S.elliptic}

We keep the notation $\Gamma=\Gamma_1$, $G=\Sym(\Gamma)$,
$S=\Z\Gamma/\!\rad$,
etc.\ from \autoref{s.Picard}, and we consider a Humbert sextic~$X$ that is
general in the sense that $\Pic(X)=S$.

\subsection{Rational curves}\label{s.rational}
To find smooth rational curves on a polarized K3 surface,
we use the well known description of the nef cone
and Vinberg's algorithm~\cite{Vinberg:polyhedron}
for computing the fundamental polyhedra.
As a step, the vectors of a given square in a
definite lattice are found by the Lenstra--Lenstra--Lov\'{a}sz lattice basis reduction
algorithm, which is implemented as \texttt{ShortestVectors} in
\texttt{GAP}~\cite{GAP4.13}. One can also use the algorithm from \cite{Shimada}.

We find that on~$X$ there are
\begin{itemize}
  \item $24$ lines (a single $G$-orbit),
  \item $9$ conics (a single $G$-orbit),
  \item no twisted cubics, and
  \item $72$ rational quartics (also a single $G$-orbit).
\end{itemize}
Thus, all lines are those constituting
the original configuration~$\Gamma$ and all conics and quartics
are those described in \autoref{rem.conics} and \autoref{rem.quartics},
respectively;
we use the notation $C$, $Q$ and the indexing introduced therein.

Taking this two steps further, we find that there are
\begin{itemize}
  \item $816 = 48+192+576$ (three $G$-orbits) rational quintics and
  \item $720 = 144+288+288 $ (also three $G$-orbits) rational sextics,
\end{itemize}
so that it hardly makes sense to study these or higher degree curves in
detail.

\begin{remark}\label{rem.degenerations}
We
emphasize that these counts, as well as the smoothness of
the conics and quartics hold for a general member of the family only.
In fact,
as long as lines are concerned,
it is the nine conics $C_{rs}$ in \autoref{rem.conics}
that solely control the smooth degenerations of
Humbert sextics. Beyond the $24$ original lines, any other line is a
component of one of these conics. Furthermore, the strata with
a fixed intersection graph of lines are labeled by
the
sets of conics that split
(or rather the $(\SG3\times\SG3)\rtimes\Z/2$ orbits
thereof, see \autoref{obs.group}).
There are eight strata
$$
\def\1{\rlap{$\scriptstyle_{12}$}}
\def\2{\rlap{$\scriptstyle^2$}}
\def\one{\{1\}}
\def\group#1{\noalign{\smallskip}\multispan4\strut
 \hss\!\!$\scriptstyle#1$\!\!\hss\cr}
\def\rule{\vrule depth4pt}
\rule\
\tconfig
 1 . . \cr
 . . . \cr
 . . . \cr
\group\one
\endtconfig
\ \rule\,\rule\
\tconfig
 1 1 . \cr
 . . . \cr
 . . . \cr
\group\one
\endtconfig
\ \rule\
\tconfig
 1 . . \cr
 . 1 . \cr
 . . . \cr
\group{\Z/2}
\endtconfig
\ \rule\,\rule\
\tconfig
 1 1 . \cr
 1 . . \cr
 . . . \cr
\group{\Z/2}
\endtconfig
\ \rule\
\tconfig
 1 . . \cr
 . 1 . \cr
 . . 1 \cr
\group{\SG3}
\endtconfig
\ \rule\
\tconfig
 1 1 . \cr
 . . 1 \cr
 . . 1 \cr
\group{(\Z/2)\2}
\endtconfig
\ \rule\,\rule\
\tconfig
 1 1 . \cr
 1 1 . \cr
 . . . \cr
\group{(\Z/2)\2}
\endtconfig
\ \rule\
\tconfig
 1 1 . \cr
 1 . 1 \cr
 . 1 1 \cr
\group{\mathfrak{D}\1}
\endtconfig
\ \rule\,,
$$
itemized by the Picard rank $\rho=16,17,18,19$. (Here, the grid represents
the $(4\times4)$-cells in \autoref{table} that index the conics, and
the $\bullet$'s stand for the
conics that split. Distinct values of $\rho$ are separated by $\|$'s.
For each stratum, we indicate the group of projective automorphisms
of a general representative,
cf.\ \autoref{prop.projective} and \autoref{s.aut.as};
it is computed as explained in \autoref{rem.discr.action}.
Symplectic automorphisms preserve $\alpha$ and $\beta$
whereas anti-symplectic ones
interchange $\alpha\leftrightarrow\beta$.)
Proof will appear in~\cite{degt.Rams:sextics}.
Degenerate Humbert sextics have many more conics, twisted cubics,
and quartics;
still, only the original nine conics may split, at most six at
a time.

\end{remark}


\subsection{Elliptic pencils with a reducible singular fiber}\label{s.elliptic}
We are mostly interested in the elliptic pencils on~$X$ with at least one
reducible singular fiber made of lines. Such fibers are induced subgraphs
of~$\Gamma$ isomorphic to an affine Dynkin diagram.

All such subgraphs are listed below, where, for each type, we indicate the
total number of subgraphs followed by that itemized by the $G$-orbits. Marked
with a $^*$ are (orbits of) pencils admitting a section,
which can always be chosen a line.
\begin{itemize}
\item
$\tA_3$: $\DEFS 162 = 18\'' + 144\''$ (two orbits);
\item
$\tA_5$: $\DEFS 1056 = 192\'' + 288\'' + 576\''$ (three orbits);
\item
$\tA_7$: $\DEFS 1512 = 72 + 144\'' + 144 + 288\'' + 288 + 576\''$ (six orbits);
\item
$\tA_{11}$: $\DEFS 48$ (one orbit);
\item
$\tD_4$: $\DEFS 360 = 72 + 288\''$ (two orbits);
\item
$\tD_5$: $\DEFS  720 = 144\'' + 576\''$ (two orbits);
\item
$\tD_6$: $\DEFS 5184 = 2\*144 + 3\*288 + 4\*576\'' + 576 + 1152\''$ ($11$ orbits);
\item
$\tD_8$: $\DEFS 1440 = 288 + 2\*576$ (three orbits);
\item
$\tE_6$: $\DEFS 3840 = 2\*96 + 192\'' + 4\*576\'' + 2\*576$ (nine orbits);
\item
$\tE_7$: $\DEFS 12672 = 3\*576\'' + 7\*576 + 4\*1152\'' + 2\*1152$ ($16$ orbits);
\item
$\tE_8$: $\DEFS 4608 = 4\*1152$ (four orbits).
\end{itemize}
We analyze but a few interesting cases, leaving the rest to the reader.
Note that the list of smooth rational curves of degree up to $6$ lets us
detect reducible fibers in the pencils of fiber degree up to $12$, i.e., all except
$\tE_7$ or $\tE_8$.

The $\tA_3$-type subgraphs are the quadrangles discussed in
\autoref{s.lines}. If $q$ is a proper quadrangle,
apart from~$q$ the pencil $\pencil(q)$ has
three reducible fibers: another proper quadrangle~$q'$ and two
$\tA_1$-type fibers made of two conics each. (We use the Dynkin diagram
notation since homologically we cannot detect the degenerations
$\mathrm{I}_2\to\mathrm{II}$ or $\mathrm{I}_3\to\mathrm{III}$; most
likely they do occur in some special members of the family.) The two quadrangles
$q$, $q'$ constitute a single cell $(\qa_r,\qb_s)$, see \autoref{s.lines},
and the four conics are $C_{uv}$, $u\ne r$, $v\ne s$, see
\autoref{rem.conics}\eqref{index.qu}.

\begin{remark}\label{rem.MW}
In particular, the Mordell--Weil group of a pencil $\pencil(q)$ has
rank $5$, proving that $\Aut(X)$ is infinite, cf.\ the discussion in
\autoref{s.aut.abstract}.
\end{remark}

If $q$ is an improper quadrangle, the pencil
$\pencil(q)$ has two more reducible fibers:
another improper quadrangle~$q'$ and an $\tA_2$-type fiber made of a
conic~$C$
and two lines $L_i$, $M_j$.
Thus, we also have an involution $q\leftrightarrow q'$ on the set of improper
quadrangles; it does \emph{not} preserve $(3,3)$-fragments.
The pencil is determined by the pair $L_i$,
$M_j$ of intersecting lines:
the conic is $C=C_{rs}$, where $\qa_r\ni L_i$, $\qb_s\ni M_j$, and $q$, $q'$
are the complementary quadrangles in the two $(3,3)$-fragments sharing
$(L_i,M_j)$ as a common corner.

Representatives of the three orbits of the type ~$\tA_5$ (hexagonal) fibers
and the other reducible fibers of the corresponding pencils are as follows:
\begin{itemize}
\item
$\DEFS\tA_5=(\l1, \l13, \l2, \l17, \l4, \l21)$:
 $\DEFS( \l3, \c5_2 )+( \l11, \c5_1 )$;
\item
$\DEFS\tA_5=(\l1, \l13, \l2, \l17, \l5, \l18)$:
 $\DEFS( \l7, \c5_3 ) + ( \l15, \c5_3 ) + ( C_{33}, Q_{62} )$;
\item
$\DEFS\tA_5=(\l1, \l13, \l2, \l17, \l5, \l24)$:
 $\DEFS( \l3, \l15, Q_{12,2} ) + ( \l7, \c5_2 )$,
\end{itemize}
where $(\cdot)$ is an $\mathrm{I}_*$-type fiber and $\DEFS\c{d}_n$ is a
certain representative of the $n$-th orbit of degree~$d$ curves.
In particular, we have a ``natural''
expression of the quintics in terms of lines. (All  sextics
appear in the singular fibers of octagonal pencils.)

An example of a longest cycle $\tA_{11}$ is
$$
\DEFS
\l1, \l13, \l6, \l17, \l10, \l22, \l3, \l16, \l7, \l19, \l11, \l24,
$$
and the corresponding pencil has no other reducible fibers.
The divisor class of a fiber lies in $4\cdot S\dual$, and all lines and conics
are $4$-fold sections.

The $\tD_4$-type subgraphs constituting the shorter, $72$-element orbit are
described as follows. Pick a quartet $\qa_s$ and a $\beta$-line $M_j$ (or
vice versa, with $\alpha$ and $\beta$ reversed).
Each quartet $\qa_u$, $u\ne s$, has two lines that intersect~$M_j$ and,
together with~$M_j$, the four lines obtained constitute the $\tD_4$ in
question. For example,
$\DEFS  \l1, \l2, \l6, \l8, \l13 $ starting from $(\qa_3,M_1)$.
The elliptic pencil has two other reducible fibers:
\begin{itemize}
  \item the complementary $\tD_4$-fragment
  $\DEFS \l3, \l4, \l5, \l7, \l16$, so that all eight $\alpha$-lines
  constitute $\alpha\smallsetminus\alpha_s$, and
  \item an $\mathrm{I}_4$-type fiber $L_{10},C_{32},L_{11},C_{33}$.
\end{itemize}
In the latter, in the invariant terms,
the two conics are $C_{sv}$, $\qb_v\not\ni M_j$,
and the two lines are
those from~$\qa_s$ that are disjoint from~$M_j$.

Certainly, one can construct numerically effective
isotropic classes by combining smooth
rational curves of higher degrees. Some of this combinations appear above
(most notably, pairs of conics), but the list is far from complete.


\begin{thebibliography}{99}


\bibitem{degt.Rams:sextics}
A.~Degtyarev,  S.~ Rams, \emph{Lines on sextic {$K3$}-surfaces with
  simple singularities}, To appear, 2025.


\bibitem{DolgachevK3} I. Dolgachev. \textit{Mirror symmetry for lattice polarized K3 surfaces}.
Algebraic geometry, {\bf 4},  J. Math. Sci. {\bf 81}, 2599--2630 (1996)

\bibitem{CAG} I.\ Dolgachev, \textit{Classical Algebraic Geometry}, Cambridge Univ. Press 2012.

\bibitem{DolgDuncan} I.~Dolgachev, A.~Duncan, \textit{Automorphisms of cubic surfaces in positive characteristic}.
Izv. Ross. Akad. Nauk Ser. Mat. {\bf 83} (2019), no. 3, 15--92.

\bibitem{DolgKondo} I. Dolgachev, S. Kond\=o, \textit{Desmic quartic surfaces in arbitrary characteristic}, arXiv:2505.01033.

\bibitem{GAP4.13}
The GAP~Group, \emph{{GAP -- Groups, Algorithms, and Programming, Version
  4.13.0}}, 2024.

\bibitem{Goryunov} V.~Goryunov, \textit{Symmetric quartics with many nodes},
Adv. Soviet Math., {\bf 21} (1994), 147--161,
American Mathematical Society.


\bibitem{Humbert} G. Humbert, \textit{Sur la surface desmique du quatri\'{e}me ordre}, J. Math. Pures et Appl., {\bf 7} (1891), 353--398.



\bibitem{Huybrechts}
D.~ Huybrechts, \emph{Lectures on {K}3 surfaces}, Cambridge Studies in
  Advanced Mathematics, vol. 158, Cambridge University Press, Cambridge, 2016.

\bibitem{Kalker}, T. Kalker, \textit{Cubic fourfolds with fifteen ordinary double points},
thesis, Leiden University, 1986.



\bibitem{Nikulin} V.\ V.\ Nikulin, \textit{Integral symmetric bilinear forms and some of
their applications}, Math. USSR-Izv., {\bf 14} (1980), 103--167.

\bibitem{Nikulin2} V.\ V.\ Nikulin, \textit{Factor groups of groups of the automorphisms
of hyperbolic forms with respect to subgroups generated by
$2$-reflections}, J. Soviet Math., {\bf 22} (1983), 1401--1475.

\bibitem{Saint-Donat} B.\ Saint-Donat, \textit{Projective models of K-3 surfaces}, Amer. J. Math., {\bf 96} (1972), 602--639.

\bibitem{Shimada} I.~Shimada, \textit{Projective models of the supersingular K3 surface with
Artin invariant 1 in characteristic 5}.
J. Algebra, {\bf  403} (2014), 273--299.

\bibitem{vGeer} G.\ van der Geer, \emph{Hilbert modular surfaces}, Springer-Verlag 1987.


\bibitem{Vinberg:polyhedron}
\`{E}.~B. Vinberg, \emph{The groups of units of certain quadratic forms}, Mat.
  Sb. (N.S.) \textbf{87(129)} (1972), 18--36.


\end{thebibliography}
\end{document}